\documentclass[11pt]{amsart}
\usepackage{amscd,amssymb,float}

\newtheorem{theorem}{Theorem}[section]
\newtheorem{lemma}[theorem]{Lemma}

\newtheorem{corollary}[theorem]{Corollary}
\newtheorem{proposition}[theorem]{Proposition}

\renewcommand{\leq}{\leqslant}
\renewcommand{\geq}{\geqslant}

\theoremstyle{definition}

\theoremstyle{definition}

\numberwithin{equation}{section}

\numberwithin{equation}{section} \numberwithin{figure}{section}

\author{}
\address{}
\address{}
\email{}
\title{Almost Non-positive  K\"{a}hler Manifolds} 
\author{Yuguang Zhang}
\address{Xi'an Jiaotong-Liverpool University, Suzhou, China}
\email{Yuguang.Zhang@xjtlu.edu.cn}

\begin{document}  
\begin{abstract}
This paper proves that the universal covering   of a compact K\"{a}hler manifold with small positive sectional curvature in a certain sense  is contractible.  
\end{abstract}

\maketitle

 \section{Introduction}
  The Cartan–Hadamard theorem shows that the universal covering of a  Riemannian manifold with non-positive sectional curvature is  the Euclidean space, which  has been generalised to the case of almost non-positive curved manifolds, i.e. manifolds with small positive curvature. More precisely,  a theorem due to  Fukaya and Yamaguchi (Theorem 16.11 in \cite{Fu} and \cite{FuY}) asserts that there exists a constant $\epsilon >0$ such that if $(M, g)$ is a Riemannian manifold and $$-1 \leq K_{g} \leq \epsilon,  \  \  \  \  \  \   {\rm diam}_g(M)\leq D,  $$ where ${\rm diam}_g(M)$ denotes the diameter and $K_{g}$ is the Riemannian sectional curvature, then 
  the universal covering
space of $M$ is diffeomorphic to the Euclidean space.  The condition  of  the curvature  lower bound   cannot be removed, and in fact, a counterexample, i.e. the existence of almost non-positive curvature  metrics  on  $S^3$, has been discovered  by  Gromov and  Buser-Gromoll. 

 This paper studies   K\"{a}hler manifolds with small positive curvature. Of course, the Fukaya-Yamaguchi theorem still holds in this case since a K\"{a}hler metric is  Riemannian.   However, we want to  replace the diameter by a more computable cohomological  quantity in K\"{a}hler geometry, and remove the hypothesis of curvature's   lower bounds.

  Let $(X,\omega_g, g)$ be a compact
$n$-dimensional K\"{a}hler manifold,
  where  $g$
is a K\"{a}hler metric and $\omega_g$ denotes the  K\"{a}hler form of $g$. Here $X$ means a smooth manifold $M$ equipped with a complex structure $J$.   If $\tilde{g}$ is another K\"{a}hler metric and
$\omega_{\tilde{g}}$ is the K\"{a}hler form associated to $\tilde{g}$, the energy of
$\tilde{g}$ with respect to the background metric $g$ is defined  by
 \begin{equation}E_{g}(\tilde{g})=\frac{1}{(n-1)!} \int_{X}\omega_{\tilde{g}}\wedge
\omega_g^{n-1}=\int_{X}e(\tilde{g})dv_{g}, \end{equation}
 where $$e(\tilde{g})=\frac{1}{2}{\rm tr}_{g}\tilde{g}={\rm tr}_{\omega_g}\omega_{\tilde{g}}.$$  Note that  $E_{g}(\tilde{g})$ depends only on the
 cohomology classes $[\omega_{\tilde{g}}]$ and $[\omega_g] \in
 H^{1,1}(X,\mathbb{R})$.  

The main result is the following theorem. 

  \begin{theorem}\label{main}   Let $(X, \omega_g,  g)$ be a compact
K\"{a}hler manifold of complex dimension $n$. There exists a constant $\epsilon =\epsilon(X, [\omega_g])>0$ depending only on the complex manifold  $X$ and the K\"{a}hler class $ [\omega_g]\in H^{1,1}(X, \mathbb{R})$ satisfying that
if there is a 
K\"{a}hler metric $\tilde{g}$ such that  \begin{equation}\label{eq0.1}
K_{\tilde{g}} E_{g}(\tilde{g}) \leq \epsilon,
\end{equation}  where $K_{\tilde{g}}$ is  the Riemannian  sectional curvature
of $\tilde{g}$, then  
\begin{itemize}
\item[(i)]  the universal covering space of $X$ is
contractible, and
\item[(ii)] the holomorphic cotangent bundle $T^{*(1,0)}X$  is numerically   effective (nef).  
\end{itemize}
\end{theorem}

Here  the Riemannian  sectional curvature $K_{g}$ is regarded as a function  $K_{g}=K_{g}(x, \xi)$ of  a point $x\in X$, and a plane $\xi \subset T_xX$.  See Definition 1.9 in \cite{DPS} for the definition of nef vector bundles. 

There are many works on  the structure of  K\"{a}hler manifolds with non-positive bisectional curvature. For instance, a conjecture of Yau, proved by Liu, Wu-Zheng, and H\"{o}ring \cite{Liu,WZ,Ho}   under various assumptions,  shows that 
 a compact  K\"{a}hler manifold  with   non-positive bisectional curvature admits a  torus fibre 
bundle structure.  These results have been generalised to   K\"{a}hler manifolds with nef cotangent bundle by H\"{o}ring \cite{Ho}.  We apply (ii) of 
 Theorem \ref{main} to Theorem 1.2 and Theorem 1.4 in  \cite{Ho}, and obtain the following 
 corollary.  
 
  \begin{corollary}\label{co+0}     Let $(X, \omega_g,  g)$, $\epsilon=\epsilon(X, g)$, and $\tilde{g}$ be the same as those in Theorem \ref{main}. If either $\dim_{\mathbb{C}} X\leq 3$, or $X$ is a projective manifold with semi-ample canonical bundle, 
   then a finite covering space $X'$  of $X$ admits a torus  fibration $X'\rightarrow Y$ onto a  K\"{a}hler manifold $Y$ of negative first Chern class, i.e. $c_1(Y)<0$.  
 \end{corollary}
 
 Note that the  fibration obtained here may not be a  fibre 
bundle since the complex structures of fibre tori could vary.  As pointed out in  \cite{Ho}, there are examples of manifolds with nef cotangent bundle but not admitting  fibre 
bundle structures, e.g. 
the total spaces of universal families over compact curves in the moduli space
of polarised abelian varieties.   Furthermore, K\"{a}hler metrics of small positive  sectional curvature are expected to exist on these manifolds. More precisely, if $X\rightarrow Y$ is a fibration over a higher genus Riemann  surface $Y$ with polarised abelian varieties as  fibres,
the techniques developed in Section 3 of  \cite{GTZ} could  be used to construct a family of K\"{a}hler metrics $g_t$, $t\in (0, 1]$, on $X$ satisfying the following. 
\begin{itemize}
\item[(i)]   $g_t$ is a  collapsing   semi-flat metric, i.e.  the restrictions of $g_t$ on fibres  are flat, and   the diameters of fibres tend to zero when $t\rightarrow 0$.
\item[(ii)]   The pull-back metric  of $g_t$ on the universal covering space of $X$ converges smoothly to $g_H+g_E$ on $\Delta \times \mathbb{C}^{n-1}$, as $t\rightarrow 0$, where $g_H$ denotes the standard hyperbolic metric on the disc $\Delta$ and $g_E$ is the Euclidean metric. 
\item[(iii)]  The condition of Theorem \ref{main} is satisfied, i.e. $$ \sup\limits_X  K_{g_t}E_g(g_t) \rightarrow 0,  \  \  \  \  t\rightarrow 0.$$
\end{itemize}
We leave the details to interested readers.

  The idea to prove (i) in Theorem  \ref{main}  is as follows. Assume that the universal covering  of $X$ is not  
contractible and there is a sequence of K\"{a}hler metrics $g_k$ with $E_g(g_k)=1$ and sectional curvature $K_{g_k}\leq \frac{1}{k}$.  A theorem due to Sacks and Uhlenbeck (Theorem 5.8 in \cite{SU}) asserts that, for each  $k>0$,  there is a non-trivial smooth conformal branched minimal immersion $u: S^2 \rightarrow X$ with respect to the metric $g_k$. We further assume that $u$ is an embedding and consider the restricted metric $g_k|_{u(S^2)}$ on the image of $u$. Since the sectional curvature of a minimal surface is smaller or equal to the sectional curvature of the ambient  space, the Gauss-Bonnet formula shows $$4\pi =2\pi \chi (S^2)=\int _{S^2} K_{u^*g_k}dv_{u^*g_k} \leq \frac{1}{k}{\rm Vol}_{u^*g_k}(S^2). $$ If we can find a uniformly   upper bound of the volume, i.e. ${\rm Vol}_{u^*g_k}(S^2)<v$ for a constant $v$ independent of $k$, then it is a contradiction by letting $k\rightarrow\infty$. For achieving the upper bound of volumes, we need a  Schwarz type inequality for  K\"{a}hler metrics with small positive curvature, i.e. in the current case,  \begin{equation}\label{eq0.2} g_k \leq \bar{C} g \end{equation} for a constant $\bar{C}$ independent of $k$, which is obtained by Proposition \ref{prop1} in Section 2. (\ref{eq0.2}) also implies (ii) of Theorem \ref{main} by combining a direct computation.  Section 3 proves  Theorem \ref{main}.

We also expect a K\"{a}hler  analogue  of Gromov's almost flat manifolds.
If  there is a sequence of  K\"{a}hler metrics $g_k$ with $$E_g(g_k)=1,  \ \  \  \  \  \  -\frac{1}{k} \leq K_{g_k} \leq\frac{1}{k} $$  on a K\"{a}hler manifold $(X, g)$, then (\ref{eq0.2}) gives an upper bound of diameters, i.e. $${\rm diam}_{g_k}(X)\leq D,$$ for a constant $D>0$ independent of $k$.  When $k\gg 1$, $(X, g_k)$ satisfies  the hypothesis of the Gromov theorem for almost flat manifolds (Theorem 8.1 in \cite{Fu} and \cite{Gro}). Therefore, the Gromov theorem implies that a finite covering $X'$ of $X$ is diffeomorphic to a nil-manifold.   $X'$ carries K\"{a}hler structures by pulling back K\"{a}hler metrics on $X$. 
However, by Theorem A in \cite{BG},
none of  nil-manifold  other than torus admits a K\"{a}hler metric.
Hence we have proved the following result, i.e. almost flat K\"{a}hler manifolds are flat.

 \begin{corollary}\label{co1}     Let $(X, \omega_g,  g)$ be a compact
K\"{a}hler manifold of complex dimension $n$. There exists a constant $\epsilon =\epsilon(X, [\omega_g])>0$ depending only on the complex manifold  $X$ and the K\"{a}hler class $ [\omega_g]\in H^{1,1}(X, \mathbb{R})$ satisfying that
if there is a
K\"{a}hler metric $\tilde{g}$ such that $$|K_{\tilde{g}}| E_{g}(\tilde{g}) \leq
\epsilon, $$ where $K_{\tilde{g}}$ is  the Riemannian  sectional curvature
of $\tilde{g}$,  then  a finite covering of
$X$ is   a  torus. 
 \end{corollary}

{\bf Acknowledgments. }  Some results are  contained    in  the author's PhD thesis \cite{Zh}, and the author would like to  thank his supervisor Professor Fuquan Fang for guidance.

\section{A   Schwarz  type inequality}

The following inequality could be regarded as a generalisation of the  Schwarz   inequality (Theorem 2 in \cite{Yau}) for the case of  small positive curvature.  

 \begin{proposition}\label{prop1}   Let $(X, \omega_g,  g)$ be a compact
K\"{a}hler manifold of complex  dimension $n$.  Then there exist   constants
$\mathcal{E}=\mathcal{E}(g)> 0$ and $C=C(g)>0$ depending only on  the K\"{a}hler   metric $g$,   such that if a
 K\"{a}hler metric  $\tilde{g}$ satisfies   $$ K_{\tilde{g}}^hE_{g}(\tilde{g}) \leq \mathcal{E}, $$
 where $K_{\tilde{g}}^h$ is the holomorphic   bisectional curvature of
$\tilde{g}$, then $$ \tilde{g} \leq CE_{g}(\tilde{g})g.$$
 \end{proposition}

 Let $(X, \omega_g,  g)$ be a compact
K\"{a}hler manifold.  
Note that the
identity map
$${\rm Id} :(X,g) \rightarrow (X,\tilde{g})$$ is a holomorphic map, and therefore,  is a harmonic map. The energy
density of Id is given by  $$ e(\tilde{g})=\frac{1}{2}
|d({\rm Id})|^{2}=\frac{1}{2}{\rm tr}_{g}\tilde{g} . $$    The Chern-Lu inequality (cf. \cite{Yau,Ch,Lu}) says
 \begin{equation} -\frac{1}{2}\Delta_g e(\tilde{g})\leq r_{c}e(\tilde{g})+ \overline{K}^he(\tilde{g})^{2},
 \end{equation} where     $- r_{c}<0$ is a  lower bound of  the Ricci curvature of $g$, i.e. ${\rm Ric}(g)\geq -r_{c}$,   $\overline{K}^h>0$ is an upper bound of  the holomorphic  bisectional curvature of $\tilde{g}$, $K^h_{\tilde{g}}\leq 
\overline{K}^h$, and $\Delta_g$ is the Laplacian operator of $g$.  
Proposition \ref{prop1} is a consequence of   a quantitative version  of the  Schoen-Uhlenbeck
small energy estimate for harmonic maps  (cf. \cite{Sc} and see also Proposition 2.1 in \cite{Ru}).  

 \begin{lemma}\label{le1}  There exist positive  constants $R(g)$,  $ \varepsilon(g)$
and $\bar{C}(g)$  depending only on the injectivity radius  and
the bound  of  curvature of $g$ such that, for any metric r-ball
$B_{g}(x,r)$ with  $r \leq R(g)$ and $x\in X$,    if
$$ \frac{r^{2}}{{\rm Vol}_{g}(B_{g}(x, r))}\int_{B_{g}(x, r)}e(\tilde{g})dv_{g}\leq
\frac{\varepsilon(g)}{ \overline{K}^h}, $$ then
$$ \sup\limits_{B_{g}(x, \frac{r}{4})}e(\tilde{g}) \leq \bar{C}(g)
\frac{1}{{\rm Vol}_{g}(B_{g}(x, r))}\int_{B_{g}(x, r)}e(\tilde{g})dv_{g}.$$ \end{lemma}

 \begin{proof}  The only difference between this lemma and  the well-understood  Schoen-Uhlenbeck
 estimate in  \cite{Sc} is that the constant $ \overline{K}^h$ in the Chern-Lu inequality   enters explicitly into  the formula of the small energy condition. To  prove this lemma,   we need only  to track where $ \overline{K}^h$ goes in  the original  proof  in    \cite{Sc}.  Here  we also consult the proof of   Theorem  2.2.1 in \cite{Tian}.   We present the details for readers' convenience.

  Firstly, there exist constants  $R(g)$ and $\Lambda$
depending only on the injectivity radius  and the bound
of  curvature of $g$ such that, on any metric $R(g)$-ball
$B_g(x, R(g))$, there is a harmonic coordinate system $\{x^1, \cdots, x^{2n}\}$ on
$B_g(x, R(g))$, i.e.  $\Delta_g x^i=0 $, satisfying $x=(0, \cdots, 0)$, 
$$\Lambda^{-1}(\delta_{ij})\leq (g_{ij})\leq \Lambda (\delta_{ij}), \ \ \ {\rm and}\ \ \   \|g_{ij}\|_{C^{1,
\frac{1}{2}}}< \Lambda,
$$ where $g_{ij}=g( \frac{\partial}{\partial x^{i}},  \frac{\partial}{\partial
x^{j}})$ (cf. Section 1  in \cite{GW}).  

 Note that there exists a  $\sigma_{0}\in [0,\frac{r}{2})$, $ r\leq R(g)$,  such
that $$ (r-2\sigma_{0})^{2}\sup\limits_{B_g(x, 
\sigma_{0})}e(\tilde{g})=
\max\limits_{0\leq\sigma \leq\frac{ r}{2}}(r-2\sigma)^{2}\sup\limits_{B_g(x, 
\sigma)}e(\tilde{g}).$$ Moreover, there exists a point
$x_{0}\in B_g(x, \sigma_{0}) $ such that $$
e_{0}=e(\tilde{g})(x_{0})=\sup\limits_{B_g(x,  \sigma_{0})}e(\tilde{g}).$$ 
If we let $\rho_{0}= \frac{1}{4} (r- 2 \sigma_{0})$, then  $B_g(x_0, \rho_{0}) \subset B_g(x, \sigma_0+\rho_0)\subset B_g(x, \frac{r}{2})$. 
We obtain 
$$\sup\limits_{B_g(x_0, \rho_{0})}e(\tilde{g})\leq \sup\limits_{B_g(x, 
\sigma_{0}+\rho_{0})}e(\tilde{g})\leq  \frac{(r-2\sigma_0)^2}{(r-2\sigma_0-2\rho_0)^2} e_0= 4e_{0}.$$

 Now we assume $
\overline{K}^he_{0}>1$. Consider the re-scaled metric $
\bar{g}=\overline{K}^he_{0}g$, and the metric ball $
B_{\bar{g}}(x_0, r_{0})$ of $\bar{g}$, where
$r_{0}=(\overline{K}^he_{0}) ^{ \frac{1}{2}}\rho_{0}$. Then the energy
density of the identity map with respect to  the rescaled  metric $\bar{g}$
reads $$ \bar{e}(\tilde{g})=
\frac{1}{\overline{K}^he_{0}}e(\tilde{g}).$$ Therefore $
\bar{e}(\tilde{g})(x_{0})=\bar{e}_0= \frac{1}{
\overline{K}^h}$, and we obtain  $$
\sup\limits_{B_{\bar{g}}  (x_0, r_{0})}\bar{e}(\tilde{g})\leq 4 \bar{e}_0=
\frac{4}{ \overline{K}^h}.$$  Since $ \bar{r_{c}}=
\frac{r_{c}}{\overline{K}^he_{0}}<r_{c} $, the  Chern-Lu inequality (\ref{eq3.1})  says
$$
-\frac{1}{2}\Delta_{\bar{g}}\bar{e}(\tilde{g})  \leq
\bar{r_{c}} \bar{e}(\tilde{g})+ \overline{K}^h
\bar{e}(\tilde{g})^{2} \leq  ( r_{c}+4) \bar{e}(\tilde{g})
$$ on $ B_{\bar{g}}  (x_0, r_{0})$, where $\bar{g}=\sum g_{ij}dy^idy^j $, 
$$\Delta_{\bar{g}}=\sum g^{ij}\frac{\partial ^{2}}{\partial y^{i}
\partial y^{j}}+ \frac{1}{ \sqrt{{\rm det}(g_{ij})}}
\frac{\partial}{\partial y^{i}}({\rm det}(g_{ij})g^{ij}
)\frac{\partial}{\partial y^{j}},$$ and $y^{i}=(\overline{K}^he_{0}) ^{
\frac{1}{2}}x^{i}$.

If $r_{0}\geq 1$, i.e. $\rho_{0}\geq ( \overline{K}^he_{0}) ^{-
\frac{1}{2}}$, then by the mean value inequality (Theorem 9.20 in
\cite{GT}) we obtain $$
\frac{1}{\overline{K}^h}=\bar{e}_{0}\leq C_{1}( r_{c}+5)
\int_{ B_{\bar{g}}(x_0,1)} \bar{e}(\tilde{g})dv_{
\bar{g}}$$ for a constant  $C_{1}$  depending
only on $\Lambda(g)$ and $n$.  Note that $$( \overline{K}^he_{0}) ^{-
\frac{1}{2}} \leq \rho_{0}=\frac{1}{4}(r-2\sigma_0) \leq \frac{r}{2}.$$  
 By  the monotonicity
inequality for harmonic maps (cf. Theorem 1" (a) in  \cite{Price}),
\begin{eqnarray*} \int_{ B_{\bar{g}}(x_0,1)} \bar{e}(\tilde{g})dv_{
\bar{g}} & = & (\overline{K}^he_{0}) ^{n-1}
\int_{B_{g}(x_{0}, (\overline{K}^h e_{0})^{-\frac{1}{2}})}e(\tilde{g})
dv_{g} \\ & \leq & \frac{C_2 r^{2}}{{\rm Vol}_{g}(B_{g}(x_0,\frac{r}{2}))}\int_{B_{g}(x_0,\frac{r}{2})}e(\tilde{g})
dv_{g}
\end{eqnarray*} for a 
constant $C_{2}$ depending only on $\Lambda(g)$ and $n$. By $\sigma_0<\frac{r}{2}$, $B_g(x_0, \frac{r}{2})\subset B_g(x,r)$,  $${\rm Vol}_g(B_g(x, r))\leq \kappa' r^{2n} \leq\kappa {\rm Vol}_g(B_g(x_0, \frac{r}{2})), $$ for constants $\kappa$ and $\kappa'$ depending only on  $\Lambda(g)$ and $n$. Thus 
\begin{eqnarray*} \frac{1}{\overline{K}^h} & \leq  &  C_{1}C_{2}(r_{c}+5) \frac{ r^{2}}{{\rm Vol}_{g}(B_{g}(x_0,\frac{r}{2}))}\int_{B_{g}(x_0,\frac{r}{2})}e(\tilde{g})
dv_{g}  \\ & \leq &
\kappa C_{1}C_{2}(r_{c}+5)\frac{r^{2}}{{\rm Vol}_{g}(B_{g}(x,r))}\int_{B_{g}(x,r)}e(\tilde{g})
dv_{g}\\  & \leq &\kappa C_{1}C_{2}(r_{c}+5)  \frac{\varepsilon(g)}{\overline{K}^h}.\end{eqnarray*} If  we choose $ \varepsilon(g)=
\frac{1}{2\kappa C_{1}C_{2}(r_{c}+5)}$, then it is a contradiction.  

Therefore we assume $r_{0}<1$. By the mean value inequality (Theorem
9.20 in [37]), we have
\begin{eqnarray*} 
\frac{1}{\overline{K}^h}=\bar{e}_{0}  & \leq & C_{1}( r_{c}+5)
r_{0}^{-2n}\int_{B_{\bar{g}}(x_0, r_0)} \bar{e}(\tilde{g}) dv_{
\bar{g}} \\ &= & C_{1}( r_{c}+5) r_{0}^{-2}
\rho_{0}^{2-2n}\int_{B_{g}(x_{0}, \rho_{0})}e(\tilde{g}) dv_{g}. 
\end{eqnarray*}  Thus
\begin{eqnarray*} 
\rho_{0}^{2}e_{0}= \frac{r_{0}^{2}}{ \overline{K}^h} & \leq & C_{1}(
r_{c}+5)  \rho_{0}^{2-2n}\int_{B_{g}(x_{0}, \rho_{0})}e(\tilde{g}) dv_{g}
\\ &\leq &
C_{1}C_{2}(r_{c}+5)\frac{ r^{2}}{{\rm Vol}_{g}(B_{g}(x_0,\frac{r}{2}))}\int_{B_{g}(x_0,\frac{r}{2})}e(\tilde{g})
dv_{g}
\end{eqnarray*}  by  the monotonicity
inequality for  harmonic maps again and $\rho_0\leq \frac{r}{2}$. Hence  $$
\max\limits_{0\leq\sigma \leq\frac{ r}{2}} (r-2\sigma)
^{2}\sup\limits_{B_{\sigma}}e(\tilde{g}) \leq 16 \rho_{0}^{2}e_{0} \leq 
C_{3}\frac{r^{2}}{{\rm Vol}_{g}(B_{g}(x,r))}\int_{B_{g}(x,r)}e(\tilde{g})
dv_{g} $$  by the same argument as above. 
We take   $\sigma= \frac{1}{4}r$ and obtain
the estimate.

If $ \overline{K}^he_{0}\leq 1$, by the Chern-Lu inequality (\ref{eq3.1}) we obtain
$$-\frac{1}{2} \Delta_g e(\tilde{g})  \leq  ( r_{c}+4) e(\tilde{g})$$ on
$B_{g}(x_{0}, \rho_{0})$. Then  the   similar arguments as above 
prove  the estimate.
 \end{proof}

  In the two dimensional case, there is a    result about the explicit values of the  constants $\varepsilon$ and $\bar{C}$.  

 \begin{lemma}[Lemma 4.3.2 in \cite{MS}]\label{le-surface} If $e:
  \mathbb{R}^{2}\supset B_{h_0}(0, r)\rightarrow \mathbb{R}$ is a
  function satisfying $$ -\triangle e \leq A e^{2}, \ \ \ e>0, \ \ \  and  \ \ \ \int_{ B_{h_0}(0, r)}e dx\leq \frac{\pi
  }{12A }, $$  for  $A>0$, then  $$ e(0) \leq \frac{8}{\pi r^{2}} \int_{ B_{h_0}(0, r)}e dx $$ where $h_0$ denotes the standard Euclidean metric and $\triangle$ is the Laplacian operator with respect to  $h_0$. 
\end{lemma}

\begin{proof}[Proof of Proposition \ref{prop1}]Let $R(g)$ and $\varepsilon(g)$ be
 the constants  appeared  in Lemma \ref{le1}, and $ v(g)=\inf\limits_{x\in
 X}{\rm Vol}_{g}(B_{g}(x, R(g)))$, where $B_{g}(x, R(g))$ is a metric $R(g)$-ball.
Set $$ \mathcal{E} (g)= \frac{
\varepsilon(g) v(g)}{(R(g)) ^{2}} .$$

Assume that there exists a
 K\"{a}hler metric  $\tilde{g}$ such that $$   \sup_X K_{\tilde{g}}^hE_{g}(\tilde{g}) \leq \overline{K}^hE_{g}(\tilde{g}) \leq \mathcal{E},$$ for a constant  $\overline{K}^h>0$.   
Then for any metric $R(g)$-ball $B_g(x, R(g))$, we have
$$\frac{R(g)^{2}}{{\rm Vol}_{g}(B_{g}(x, R(g)))}\int_{B_{g}(x, R(g))}e(\tilde{g})dv_{g} \leq
\frac{R(g) ^{2}E_{g}(\tilde{g})}{v(g)} = \frac{\varepsilon(
g)E_{g}(\tilde{g})}{\mathcal{E} (g)}\leq\frac{\varepsilon(
g)}{\overline{K}^h}.$$ Since   the  Chern-Lu inequality (\ref{eq3.1}) holds,    Lemma
 \ref{le1} implies 
\begin{eqnarray*} \sup\limits_{B_{g}(x, \frac{R(g)}{4})}e(\tilde{g}) & \leq  & C(g)
\frac{1}{{\rm Vol}_{g}(B_{g}(x, R(g)))}\int_{B_{g}(x, R(g))}e(\tilde{g})dv_{g}
\\ & \leq & \frac{C(g)E_{g}(\tilde{g})}{v(g)} \\&
=& \overline{C}E_{g}(\tilde{g}),
\end{eqnarray*} where $\overline{C}$ depends only on the injectivity
radius  and the bound  of the curvature of $g$.  Therefore 
$$ \tilde{g}\leq \overline{C}E_{g}(\tilde{g})g $$ on $X$.  \end{proof}

There is an  application of Proposition \ref{prop1} to the Gromov-Hausdorff convergence of K\"{a}hler manifolds via Ruan's work \cite{Ru}.

 \begin{corollary}\label{co2}     Let $(X, \omega_g,  g)$ be a compact
K\"{a}hler manifold of complex dimension $n$. Assume that there is  a sequence of K\"{a}hler metrics $g_k$ with bounded Riemannian  curvature $$|K_{g_k}|\leq 1, \ \  \  \  and \  \ \ \ 0< \tau \leq 2 E_{g}(g_k) \leq \mathcal{E}, $$ where $\mathcal{E}$ is the constant in 
Proposition \ref{prop1}, $\tau$ is a constant independent of $k$,  and if $n=1$, $\mathcal{E}=\frac{\pi}{12}$.  Then the following holds: 
\begin{itemize}
\item[(i)] If the volume $${\rm Vol}_{g_k}(X)\geq v,$$ for a constant $v>0$ independent of $k$, i.e. the non-collapsing case, then a subsequence of  $(X, g_k)$ converges to a compact  $C^{1,\alpha}$-K\"{a}hler manifold $(Y, g_\infty)$ of the same dimension in the $C^{1,\alpha}$ Cheeger-Gromov sense. Furthermore, $X$ is biholomorphic to $Y$.   
\item[(ii)] If $${\rm Vol}_{g_k}(X)\rightarrow 0, \  \  \  when \  \  k\rightarrow\infty,$$ i.e. the   collapsing case, then   $X$ admits a nontrivial holomorphic foliation, i.e. $X$ is not the leaf. 
\end{itemize}
 \end{corollary}

\begin{proof}  Note that the bound of Riemannian sectional curvature implies the holomorphic  bisectional curvature $K_{g_k}^h\leq 2$.   Since  Proposition \ref{prop1} implies  $g_k\leq Cg$,  the blow-up subvariety in Proposition 3.1 of \cite{Ru}, defined by the non-trivialness of the Lelong number of the limit current of $g_k$, is empty. Therefore the case of non-collapsing  follows from Theorem 1.2 in  \cite{Ru}.     In the  collapsing case,  the K\"{a}hler forms $\omega_{g_k}$ converges to a non-zero  current $\omega_\infty$ in the distribution sense, and $\omega_\infty^n\equiv 0$ by  Theorem 1.2 in  \cite{Ru}. Furthermore,  Theorem 1.3 of  \cite{Ru}  shows that $\omega_\infty$ induces a nontrivial  holomorphic foliation on $X$. 
\end{proof}

 \section{Proofs }

To prove  Theorem \ref{main}, we need   a  quantitative  version  of Theorem 3.3
in \cite{SU}.  

 \begin{lemma}\label{le-gap} If $u$ is a non-trivial harmonic map from
$(S^{2},h_{1})$ to $(X,  g)$, then $$ 
\int_{S^{2}}e(u)dv_{h_{1}} \geq \frac{\pi}{24
\overline{K}},  \ \  \   \  \  \   \   \   e(u)=|du|^2={\rm tr}_{h_{1}}(u^{*} g),$$  where $h_{1}$ is the metric of Gaussian curvature 
one,  and
$\overline{K}$ is a positive  upper bound of the Riemannian  sectional curvature
of $g$. \end{lemma}

 \begin{proof} As shown in the introduction, if $u$ is a conformal minimal embedding, then the Gauss-Bonnet formula gives the lower bound of the energy. Now we prove the general case.

  Let $\varphi$ be the conformal equivalence from
$(\mathbb{R}^{2},h_{0})$ to $( S^{2}\backslash \{$the south pole$\},h_{1})$ where $h_{0}$ is the flat metric, and
$\widetilde{u}=u\circ \varphi$. Let   $(y^1, y^2)$ and
$(x^{1}, \cdots, x^{2n})$ be coordinates  on $\mathbb{R}^{2}$ and $X$ respectively
such that $h_{0}=d(y^1)^{2}+d(y^2)^{2}$ and $\varphi^{*}h_{1}=\lambda
(d(y^1)^{2}+d(y^2)^{2})$, $\lambda (y^1,y^2)>0$. $\widetilde{u}$ is also 
  a non-trivial harmonic map from $(\mathbb{R}^{2},h_{0})$ to $(X,  g)$.  
     
    The Bochner formula for harmonic
maps (cf. \cite{Sc}) says $$ \frac{1}{2}\triangle e(\widetilde{u})= |
\nabla^g d\widetilde{u} |^{2}- \sum\limits_{\mu, \nu} g(
R^{g}( \widetilde{u}_{*}\theta_\mu,
\widetilde{u}_{*}\theta_{\nu})\widetilde{u}_{*}\theta_{\mu},\widetilde{u}_{*}\theta_{\nu})
$$ where $\theta_{\mu}=\frac{\partial}{\partial y^\mu}$,   $\triangle= \frac{\partial^{2}}{\partial (y^1)^{2}}+\frac{\partial^{2}}{\partial
   (y^2)^{2}}$,   and
$e(\widetilde{u})={\rm tr}_{h_{0}}(\widetilde{u}^{*} g)$. By 
Corollary  1.7 in \cite{SU}, the harmonic map $u$ is automatically
conformal since the domain is $S^2$, and therefore  is $\widetilde{u}$.  By  $\theta_1\bot \theta_2$, $\widetilde{u}_{*}\theta_1 $ and $\widetilde{u}_{*}\theta_2$ are perpendicular with respect to $g$.  Thus $$\sum\limits_{\mu, \nu} g(
R^{g}( \widetilde{u}_{*}\theta_\mu,
\widetilde{u}_{*}\theta_{\nu})\widetilde{u}_{*}\theta_{\nu},\widetilde{u}_{*}\theta_{\mu}) \leq  \overline{K}2|\widetilde{u}_{*}\theta_1|_g^2|\widetilde{u}_{*}\theta_2|_g^2\leq\overline{K}e(\widetilde{u})^{2}. $$
We obtain
$$  -\frac{1}{2}\Delta e(\widetilde{u})\leq
\overline{K}e(\widetilde{u})^{2}. $$ 

 If
$$\int_{\mathbb{R}^{2}}e( \widetilde{u})dv_{h_{0}}  =\int_{S^{2}}e(u)dv_{h_{1}} \leq \frac{\pi}{24 \overline{K}},$$ then  Lemma
\ref{le-surface} shows $$\sup\limits_{B_{g}(0, \frac{R}{4})}
e(\widetilde{u}) \leq C \frac{1}{\pi R^{2}}
\int_{B_{g}(0, R)}e(\widetilde{u})dv_{h_{0}} < C'\frac{1}{\pi
R^{2}},$$ for any $R>0$ and constants $C$ and $C'>0$.  
 By letting $R\longrightarrow \infty$, we obtain
 $e(\widetilde{u})\equiv 0$ on $\mathbb{R}^{2}$. It is a
 contradiction.   \end{proof}
 
 The lower bound formula for holomorphic spheres has been used to study singularities of the K\"{a}hler-Ricci flow in \cite{TZ}.

 \begin{proof}[Proof of (i) in Theorem  \ref{main}] Let $(X, \omega_g, g)$ be as in Theorem  \ref{main}.   Firstly, we claim that  there exists a constant $\tilde{\epsilon}= \tilde{\epsilon} (X, g)>0$ depending only on the complex structure and the K\"{a}hler metric such that  if there is another 
K\"{a}hler metric $\tilde{g}$ satisfying  $$ 
K_{\tilde{g}} E_{g}(\tilde{g}) \leq \tilde{\epsilon},
$$  then the universal covering space of $X$ is
contractible. Secondly,  we let $$\epsilon =\frac{1}{2} \min \big\{\sup_{g' \  {\rm with} \ \omega_{g'}\in[\omega_g]}\tilde{\epsilon}(X, g), \  \   3\times 10^8 \big\},$$ which depends only on the complex manifold and the K\"{a}hler  class $[\omega_g]$. 

Assume that  this claim is true. If there is  a 
K\"{a}hler metric $\tilde{g}$ satisfying  $$ 
K_{\tilde{g}} E_{g}(\tilde{g}) \leq \epsilon \leq  \tilde{\epsilon}(X, \hat{g})\leq \sup_{g' \  {\rm with} \ \omega_{g'}\in[\omega_g]}\tilde{\epsilon}(X, g),
$$ for a K\"{a}hler metric $\hat{g} $ with the K\"{a}hler form  $ \omega_{\hat{g}}\in[\omega_g]$, then we  prove (i) of  Theorem \ref{main} by $ E_{g}(\tilde{g})= E_{\hat{g}}(\tilde{g})$ and applying the claim to  $(X, \omega_{\hat{g}}, \hat{g})$. 

Now we prove the claim. 
 Assume that the universal covering space of
$X$ is not contractible, and  there is a
sequence of K\"{a}hler metrics $\{\tilde{g}_{k}\}$ such that $ K_{\tilde{g}_{k}}
E_{g}(\tilde{g}_{k}) < \frac{1}{k}$. We rescale the metrics, let 
$g_{k}=\frac{1}{E_{g}(\tilde{g}_{k})}\tilde{g}_{k}$, and   obtain 
$$ 
  E_{g}(g_{k})=1, \  \  \  \  \  \  \  K_{g_{k}}=K_{\tilde{g}_{k}}E_{g}(\tilde{g}_{k}) <
\frac{1}{k}.
$$ Since the holomorphic bisectional curvature can be written as the sum of two sectional curvatures, the holomorphic  bisectional curvature $ K_{g_{k}}^h <
\frac{2}{k}$. 
 By Proposition \ref{prop1}, there exists  a constant $\overline{C}>0$ independent of $k$ such that
\begin{equation}\label{eq3.1} g_{k}\leq \overline{C}g.\end{equation}

  For a fixed  $k$, we consider the $\alpha$-energy of Sacks
and Uhlenbeck \cite{SU}, and follow the  arguments in  the proof
of Theorem 2.7 in \cite{SU}.  The task is to use (\ref{eq3.1}) to give an upper bound of the volume of minimal spheres obtained by Theorem 5.8 in \cite{SU}.

For each $2>\alpha
>1$, the $\alpha$-energy is a real-valued $C^{2}$ function defined
on the Banach manifold $L^{2\alpha}_{1}(S^{2}, X)_{k}\subset
C^{0}(S^{2}, X)$ of $L^{2\alpha}_{1}$ Sobolev  mappings from $(S^{2}, h_{1})$ to
$(X, g_{k})$, $$
E_{\alpha,k}(u)=\int_{S^{2}}(1+
e_{k}(u))^{\alpha}dv_{h_{1}},$$ where
$e_{k}(u)={\rm tr}_{h_{1}}(u^{*} g_{k})$ and $h_{1}$ is the standard spherical  metric   on $S^{2}$.
We take base points $x_0\in X$ and $y_0\in S^2$, and 
denote  $\Omega(S^{2},X)$  the space of base point-preserving maps
from $S^{2}$ to $X$.  The map $ C^{0}(S^{2}, X) \rightarrow X $ given by $u \mapsto u(y_0)$ defines a 
 fibration structure 
$$\Omega(S^{2}, X)\hookrightarrow C^{0}(S^{2}, X) \rightarrow X$$ with fibre $\Omega(S^{2}, X)$. If $V$ is the volume of $(S^2, h_1)$, i.e. $V=\int_{S^{2}}dv_{h_{1}}$,
 then $E_{\alpha,k}^{-1}(V)$ is the set of trivial maps, i.e. the images are single points.  The fibration  admits a section $$X \rightarrow E_{\alpha,k}^{-1}(V)\subset C^{0}(S^{2},X),  \  \ \  \  \  x\mapsto u_x,$$ where $u_x(S^2)=\{x\}$. 
  Hence the long exact sequence of homotopy groups splits, i.e. 
 $$\pi_{m}(C^{0}(S^{2},X))=\pi_{m}(X )\oplus
\pi_{m}(\Omega(S^{2}, X)), $$ for any $m$. 
 Since  we have assumed that
the universal covering space of $X$ is not contractible,
$$\pi_{m+2}(X)=\pi_{m}(\Omega(S^{2}, X))\neq \{0\}$$ for some $m\geq 0$.
And we  identify $X$ with the set of trivial maps,
$E_{\alpha,k}^{-1}(V)$. An useful fact is that the homotopy type is
the same for all mapping spaces, from $C^{0}(S^{2}, X)$ to
$C^{\infty}(S^{2}, X)$ to $L^{2\alpha}_{1}(S^{2}, X)_{k}$ (See \cite{SU}). 

If $\pi_{0}(C^{0}(S^{2}, X))\neq \{0\}$, let $\mathcal{C}\subset
C^{0}(S^{2}, X)$ be  a path connected component not containing
$E_{\alpha,k}^{-1}(V)$. Since $E_{\alpha,k}$ satisfies the
Palais-Smale condition (C) (cf.  Theorem 2.1 in \cite{SU}), it achieves 
its minimum in every component of $L^{2\alpha}_{1}(S^{2}, X)_{k}$ by Theorem 2.2 in \cite{SU}.    By
Proposition 2.3 in \cite{SU}, the critical maps lie in
$C^{\infty}(S^{2}, X)$. In $\mathcal{C}$, we locate a
differentiable map $\hat{u}$, and let $$B=\max\limits_{S^{2}}
{\rm tr}_{h_{1}}(\hat{u}^{*} g).$$ 
By (\ref{eq3.1}), we obtain
$$ e_{k}(\hat{u}) \leq  \overline{C}\max\limits_{S^{2}}
{\rm tr}_{h_{1}}(\hat{u}^{*} g) =  \overline{C}B,$$ and then
$$ \min\limits_{\mathcal{C}} E_{\alpha,k}\leq E_{\alpha,k}(\hat{u})\leq
(1+\max\limits_{S^{2}} e_{k}(\hat{u}))^{\alpha}V \leq (1+
\overline{C}B)^{\alpha}V\leq (1+ \overline{C}B)^{2}V. $$
 Let $u_{\alpha,k}\in
C^{\infty}(S^{2}, X)$ be a  critical  map which minimizes in
$\mathcal{C}$. Then the energy of it satisfies
$$ V<E_{1,k}(u_{\alpha,k})=V+\int_{S^{2}}e_{k}(u_{\alpha,k})dv_{h_{1}}\leq E_{\alpha,k}(u_{\alpha,k})
\leq  (1+ \overline{C}B)^{2}V.$$ Thus there exists a
constant $\widehat{C}$ independent of $\alpha$ and $k$ such that
\begin{equation}\label{eq3.2}0<\int_{S^{2}}e_{k}(u_{\alpha,k})dv_{h_{1}}\leq\widehat{C}.\end{equation}

If $\pi_{0}(C^{0}(S^{2}, X))= \{0\}$, we choose a non-zero homotopy
class $[\gamma ]\in \pi_{m}(\Omega(S^{2}, X))$. Note that $\gamma :
S^{m} \rightarrow  \Omega(S^{2}, X)$ has its image lying in
$C^{0}(S^{2}, X)$ and is not homotopic to any map
$\widetilde{\gamma}:S^{m}\rightarrow E_{\alpha,k}^{-1}(V)$.
 In
fact,  we can assume that, for any $z\in S^{m}$, $\gamma(z)$ is
differentiable, and depends continuously on $z$. If we   let
$$B=\max\limits_{z'\in S^{m}, y\in S^{2}}
{\rm tr}_{h_{1}}((\gamma(z'))^{*}g )(y),$$ then, for all $z\in S^{m}$,   (\ref{eq3.1}) 
implies  $$ e_{k}(\gamma(z))\leq \overline{C}\max\limits_{ S^{m},  S^{2}}
{\rm tr}_{h_{1}}((\gamma)^{*}g )  = \overline{C}B $$  and 
 $$ E_{\alpha,k}(\gamma(z))\leq
(1+\max\limits_{S^{2}} e_{k}(\gamma(z)))^{\alpha}V \leq (1+
\overline{C}B)^{\alpha}V.$$   

If
$E_{\alpha,k}$ has no critical value in $(V,(1+
\overline{C}B)^{\alpha}V)$, by Theorem 2.2 and Theorem 2.6 in \cite{SU}
there exists a deformation retraction $$\rho: E_{\alpha,k}^{-1}([V,(1+
\overline{C}B)^{\alpha}V]) \rightarrow E_{\alpha,k}^{-1}(V).$$
Then $\rho \circ \gamma : S^{m}\rightarrow E_{\alpha,k}^{-1}(V)$
is homotopic to $\gamma$, which is a contradiction.  
 Let
$u_{\alpha,k}\in C^{\infty}(S^{2},M)$ be a critical  map such that
$$ V< E_{\alpha,k}(u_{\alpha,k})\leq
 (1+
\overline{C}B)^{\alpha}V\leq (1+ \overline{C}B)^{2}V.$$
Then the energy of it satisfies
$$V<E_{1,k}(u_{\alpha,k})=V+\int_{S^{2}}e_{k}(u_{\alpha,k})dv_{h_{1}}\leq E_{\alpha,k}(u_{\alpha,k})
\leq  (1+ \overline{C}B)^{2}V.$$ 

In both cases, either $C^{0}(S^{2}, X)$ is connected or not, we obtain a critical map $u_{\alpha,k}$ with uniformly bounded  energy, i.e. 
\begin{equation}\label{eq3.3} 0<\int_{S^{2}}e_{k}(u_{\alpha,k})dv_{h_{1}}\leq\widehat{C}\end{equation} for  a
constant $\widehat{C}$ independent of $\alpha$ and $k$. 

 Now by Theorem 4.7 in \cite{SU}, if $\sup\limits_{S^{2}}
e_{k}(u_{\alpha,k})$  is uniformly bounded in $\alpha$,
$u_{\alpha,k}$ $C^{1}$-converges to a harmonic map
$u_{k}:(S^{2},h_{1})\rightarrow  (X, g_{k})$. If
$\sup\limits_{S^{2}} e_{k}(u_{\alpha,k})$  is unbounded in $\alpha$,
then there exists a non-trivial harmonic map
$u_{k}:(S^{2},h_{1})\rightarrow  (X, g_{k})$.   Moreover,
$$0<\int_{S^{2}}e_{k}(u_{k})dv_{h_{1}}\leq \limsup\limits_{\alpha \rightarrow
1}
\int_{S^{2}}e_{k}(u_{\alpha,k})dv_{h_{1}}\leq\widehat{C},$$
in both cases.
By Lemma \ref{le-gap} and $K_{g_{k}} \leq \frac{1}{k}$, we obtain
$$  k \frac{\pi}{24}\leq \int_{S^{2}}e_{k}(u_{k})dv_{h_{1}}\leq
\widehat{C}.$$   When $k\gg 1$, it is  a
contradiction.  We have proved the claim and therefore also (i) in Theorem \ref{main}.  
\end{proof}

\begin{proof}[Proof of (ii) in   Theorem  \ref{main}]
Assume that the holomorphic cotangent bundle $T^{*(1,0)}X$ of $X$ is not nef, and there is a sequence of K\"{a}hler metrics $g_k$ such that $$E_g(g_k)=1, \  \  \ {\rm and} \ \  \  K_{g_k}\leq \frac{1}{k}.  $$ Proposition \ref{prop1} holds, and thus $$g_k \leq \bar{C}g $$ for a constant $\bar{C}>0$.  

We regard $g_k$ as an Hermitian metric on the vector bundle  $T^{(1,0)}X$.  
If $z_1, \cdots, z_n$ are local normal coordinates on $X$ at $x$ and $\phi_1=\partial/\partial z_i, \cdots, \phi_n=\partial/\partial z_n$ are   orthonormal  frames of $T^{(1,0)}_xX$ with respect to $g_k$, then the curvature operator   of $g_k$ reads $$\Theta_{g_k}(T^{(1,0)}X)=\sum R_{\mu\bar{\nu}\lambda\bar{\upsilon}}dz_\mu \wedge d\bar{z}_\nu \otimes \phi_\lambda^\ast  \otimes \phi_\upsilon, $$ which is an hermitian  $(1,1)$-form with values in Hom$(T^{(1,0)}X, T^{(1,0)}X)$.  
Since the  holomorphic bisectional curvature can be written as sum of two sectional curvatures,  we have  $K^h_{g_k}\leq \frac{2}{k}$ and
\begin{eqnarray*} g_k(\langle \Theta_{g_k}(T^{(1,0)}X), \xi \wedge\bar{\xi} \rangle \zeta, \zeta) & =  & \sum R_{\mu\bar{\nu}\lambda\bar{\upsilon}}\xi_\mu\bar{\xi}_\nu \zeta_\lambda \bar{\zeta}_\upsilon 
\\ & \leq & \frac{2}{k}|\xi|_{g_k}^2|\zeta|_{g_k}^2 \\ & \leq &  \frac{2\bar{C}}{k}|\xi|_{g}^2|\zeta|_{g_k}^2,
\end{eqnarray*}
 for any two vectors $\xi$ and $\zeta\in T^{(1,0)}_xX$.  Therefore  $$\sqrt{-1}\Theta_{g_k}(T^{(1,0)}X) \leq   \frac{2\bar{C}}{k}  \omega_g \otimes {\rm  Id}_{T^{(1,0)}X}$$  in the sense of Griffiths.   $g_k$ induces an Hermitian metric $g_k^*$ on the holomorphic cotangent bundle $T^{*(1,0)}X$, and hence  a metric  on the symmetric power $S^mT^{*(1,0)}X$ for any $m\geq 1$. 
  The curvature of the induced metric   $$\sqrt{-1}\Theta_{(g_k^*)^{\otimes m}}(S^mT^{*(1,0)}X)\geq -  \frac{2\bar{C}}{k} m\omega_g \otimes {\rm  Id}_{S^mT^{*(1,0)}X},$$     and   $T^{*(1,0)}X$ is nef by Theorem 1.12 of \cite{DPS}. It is a contradiction.  
\end{proof}


\begin{thebibliography}{99}
\bibitem{BG}  C.  Benson,  C. Gordon, {\em K\"{a}hler and symplectic structures on nilmanifolds}, Topology,
27  (1988),  513--518.
\bibitem{Ch}  S. S. Chern,  {\em On holomorphic mappings of Hermitian manifolds of the same dimension}, in Proc. Symp. Pure Math.,  11, Amer. Math. Soc., (1968),  157--170.
\bibitem{DPS}  J.-P. Demailly, T. Peternell,  M. Schneider,  {\em Compact complex manifolds
with numerically effective tangent bundles}, J. Algebraic Geometry, 3 (1994),
 295--345. 
\bibitem{Fu} K. Fukaya, {\em  Hausdorff convergence of Riemannian manifolds and its application}, in Advance
Studies in Pure Mathematics,  18 (1990), 143--234.  
\bibitem{FuY} K. Fukaya,  T. Yamaguchi,  {\em Almost nonpositively curved manifolds}, J. Differential   Geometry, 
33 (1991), 143--234.  
\bibitem{GT}  D.  Gilbarg,    N. Trudinger,  {\em  Elliptic partial differential equations of second two}, Springer-Verlag,
(1983).
\bibitem{GW}  R. Green,  H.  Wu, {\em Lipschitz converges of Riemannian manifolds}, Pacific J. Math., 131
(1988),  119--141.
\bibitem{Gro}  M. Gromov,   {\em Almost flat manifolds},  J. Differential Geometry, 13 (1978), 231--241.
\bibitem{GTZ}  M. Gross,  V. Tosatti, Y. Zhang,   {\em Collapsing of abelian fibered Calabi–Yau manifolds},
Duke Math. J., 162 (2013), 517--551.
\bibitem{Ho}  A. H\"{o}ring,  {\em Manifolds with nef cotangent bundle},  Asian  J.   Math.,   17 (2013),  561--568.  
\bibitem{Liu}  G.  Liu,  {\em Compact K\"{a}hler manifolds with nonpositive bisectional
curvature},   Geom. Funct. Anal. Vol. 24 (2014), 1591--1607. 
\bibitem{Lu}   Y.  Lu,  {\em Holomorphic mappings of complex manifolds},  J. Differential Geometry, 2 (1968), 299--312.
\bibitem{MS}  D. McDuff,  D. Salamon, {\em J-holomorphic curves and symplectic topology}, Amer. Math.
Soc. Colloq. Publ. 52, Amer. Math. Soc. (2004).  
\bibitem{Price} P.  Price,  {\em  A monotonicity formula for Yang-Mills fields},  Manuscripta Math. 43 (1983), 131--166.
\bibitem{Ru}  W.-D. Ruan,  {\em On the convergence and collapsing of K\"{a}hler metrics},  J. Differential Geometry, 
52 (1999), 1--40.
\bibitem{SU} J. Sacks,  K. Uhelenbeck,  {\em The existence of minimal immersions of 2-spheres}, Annals of Mathematics, 113 (1981), 1--24. 
\bibitem{Sc}  R. Scheon,  {\em  Analytic aspects of the harmonic maps problem},  Sci. Res. Inst. Publ. 2 Springer-Verlag,  (1994), 321--358.
\bibitem{Tian} G.  Tian,  {\em Gauge theory and calibrated geometry  I}, Annals of Mathematics, 151 (2000),  193--268.  
\bibitem{TZ}  V. Tosatti, Y. Zhang,  {\em Infinite-time singularities of the K\"{a}hler–Ricci flow},  Geometry and Topology, 19 (2015), 2925--2948. 
\bibitem{WZ}   H. Wu,  F. Zheng, {\em Compact K\"{a}hler manifolds with nonpositive bisectional
curvature}, J. Differential Geometry, 61 (2002),  263--287
\bibitem{Yau}  S.-T.  Yau,  {\em A General Schwarz Lemma for K\"{a}hler Manifolds},  American Journal of Mathematics, 100 (1978), 
197--203. 
\bibitem{Zh}  Y. Zhang, {\em Convergence of K\"{a}hler manifolds and calibrated fibrations}, PhD thesis,
Nankai Institute of Mathematics (2006).  
\end{thebibliography}
\end{document}